\documentclass[reqno]{amsart}

\usepackage{amssymb}
\usepackage{graphicx}
\usepackage{amscd}

\usepackage[pagebackref]{hyperref}%\usepackage[hidelinks]{hyperref}
\usepackage{color}
\usepackage{float}
\usepackage{graphics,amsmath,amssymb}
\usepackage{amsthm}
\usepackage{amsfonts}
\usepackage{latexsym}
\usepackage{epsf}
\usepackage{xifthen}
\usepackage{mathrsfs}
\usepackage{dsfont}
\usepackage{makecell}
\usepackage[FIGTOPCAP]{subfigure}
\usepackage{amsmath}
\allowdisplaybreaks[4]
\usepackage{listings}
\usepackage{etoolbox}
\usepackage{fancyhdr}
\usepackage{pdflscape}
\usepackage[title,toc,titletoc]{appendix}
\usepackage{enumitem}
\usepackage[noadjust]{cite}

 \newtheoremstyle{mytheorem}% name % cf. thmtest.tex of AMSLaTeX
 {3pt}%      Space above
 {3pt}%      Space below
 {\slshape}% Body font
 {}%         Indent amount (empty = no indent,
 {\bfseries}% Thm head font
 {.}%        Punctuation after thm head
 { }%        Space after thm head: " " = normal interword space;
 {}%         Thm head spec (can be left empty, meaning `normal')

\numberwithin{equation}{section}

\theoremstyle{plain}
\newtheorem{theorem}{Theorem}[section]
\newtheorem*{theorem*}{Theorem}

\newtheorem{corollary}[theorem]{Corollary}
\newtheorem{lemma}[theorem]{Lemma}

\newtheorem{conjecture}{Conjecture}[section]

\newtheorem{innercustomgeneric}{\customgenericname}
\providecommand{\customgenericname}{}
\newcommand{\newcustomtheorem}[2]{%
	\newenvironment{#1}[1]
	{%
		\renewcommand\customgenericname{#2}%
		\renewcommand\theinnercustomgeneric{##1}%
		\innercustomgeneric
	}
	{\endinnercustomgeneric}
}
\newcustomtheorem{ctheorem}{Theorem}
\newcustomtheorem{clemma}{Lemma}

\theoremstyle{definition}
\newtheorem{definition}{Definition}[section]

\newtheorem{example}{Example}[section]
\newtheorem*{example*}{Example}
\newtheorem*{examples*}{Examples}
\newtheorem{remark}{Remark}[section]
\newtheorem*{remark*}{Remark}
\newtheorem*{remarks*}{Remarks}

\newtheoremstyle{named}{}{}{\itshape}{}{\bfseries}{.}{.5em}{#1\thmnote{ #3}}
\theoremstyle{named}

\newcommand{\doi}[1]{\href{http://dx.doi.org/#1}{doi: #1}}

\newcommand{\Keywords}[1]{\ifthenelse{\isempty{#1}}{}{\smallskip \smallskip \noindent \textbf{Keywords}. #1}}
\newcommand{\MSC}[2][2010]{\ifthenelse{\isempty{#2}}{}{\smallskip \smallskip \noindent \textbf{#1MSC}. #2}}
\newcommand{\abstractnote}[1]{\ifthenelse{\isempty{#1}}{}{\smallskip \smallskip \noindent \textsuperscript{\dag}#1}}

\makeatletter
\def\specialsection{\@startsection{section}{1}%
  \z@{\linespacing\@plus\linespacing}{.5\linespacing}%
  {\normalfont}}% NEW
\def\section{\@startsection{section}{1}%
  \z@{.7\linespacing\@plus\linespacing}{.5\linespacing}%
  {\normalfont\scshape}}% NEW
\patchcmd{\@settitle}{\uppercasenonmath\@title}{\Large\boldmath}{}{}
\patchcmd{\@settitle}{\begin{center}}{\begin{flushleft}}{}{}
\patchcmd{\@settitle}{\end{center}}{\end{flushleft}}{}{}
\patchcmd{\@setauthors}{\MakeUppercase}{\normalsize}{}{}
\patchcmd{\@setauthors}{\centering}{\raggedright}{}{}
\patchcmd{\section}{\scshape}{\large\bfseries\boldmath}{}{}
\patchcmd{\subsection}{\bfseries}{\bfseries\boldmath}{}{}
\renewcommand{\@secnumfont}{\bfseries}
\patchcmd{\@startsection}{\@afterindenttrue}{\@afterindentfalse}{}{}
\patchcmd{\abstract}{\leftmargin3pc}{\leftmargin1pc}{}{}

\def\maketitle{\par
  \@topnum\z@ % this prevents figures from falling at the top of page 1
  \@setcopyright
  \thispagestyle{empty}% this sets first page specifications
  \ifx\@empty\shortauthors \let\shortauthors\shorttitle
  \else \andify\shortauthors
  \fi
  \@maketitle@hook
  \begingroup
  \@maketitle
  \toks@\@xp{\shortauthors}\@temptokena\@xp{\shorttitle}%
  \toks4{\def\\{ \ignorespaces}}% defend against questionable usage
  \edef\@tempa{%
    \@nx\markboth{\the\toks4
      \@nx\MakeUppercase{\the\toks@}}{\the\@temptokena}}%
  \@tempa
  \endgroup
  \c@footnote\z@
  \@cleartopmattertags
}
\makeatother

\newcommand{\fS}{\mathfrak{S}}

\newcommand{\bs}{\mathbf{s}}
\newcommand{\bN}{\mathbf{N}}
\newcommand{\bI}{\mathbf{I}}
\newcommand{\bSL}{\mathbf{SL}}
\newcommand{\bLS}{\mathbf{LS}}

\newcommand{\asc}{\mathsf{asc}}
\newcommand{\des}{\mathsf{des}}
\newcommand{\ides}{\mathsf{ides}}
\newcommand{\exc}{\mathsf{exc}}
\newcommand{\dist}{\mathsf{dist}}
\newcommand{\iar}{\mathsf{iar}}
\newcommand{\tig}{\mathsf{tig}}
\newcommand{\alt}{\mathsf{alt}}
\newcommand{\lar}{\mathsf{lar}}
\newcommand{\sma}{\mathsf{sma}}
\newcommand{\LMI}{\mathsf{LMI}}
\newcommand{\rma}{\mathsf{rma}}
\newcommand{\iasc}{\mathsf{iasc}}

\newcommand{\RLAR}{\mathsf{R}_{\mathsf{lar}}}
\newcommand{\RSMA}{\mathsf{R}_{\mathsf{sma}}}

\newcommand{\cA}{\mathcal{A}}
\newcommand{\cN}{\mathcal{N}}
\newcommand{\cSL}{\mathcal{S}}
\newcommand{\cLS}{\mathcal{L}}
\newcommand{\cG}{\mathcal{G}}
\newcommand{\cF}{\mathcal{F}}
\newcommand{\cI}{\mathcal{I}}
\newcommand{\cC}{\mathcal{C}}
\newcommand{\cP}{\mathcal{P}}
\newcommand{\cH}{\mathcal{H}}

\newcommand{\coSL}{S}
\newcommand{\coLS}{L}
\newcommand{\coN}{N}

\title[Restricted Eulerian distributions]{Burstein's permutation conjecture, Hong and Li's inversion sequence conjecture, and restricted Eulerian distributions}

\author[S. Chern]{Shane Chern}
\address[S. Chern]{Department of Mathematics and Statistics, Dalhousie University, Halifax, NS, B3H 4R2, Canada}
\email{chenxiaohang92@gmail.com; xh375529@dal.ca}

\author[S. Fu]{Shishuo Fu}
\address[S. Fu]{College of Mathematics and Statistics, Chongqing University, Huxi campus, Chongqing 401331, P.R. China}
\email{fsshuo@cqu.edu.cn}

\author[Z. Lin]{Zhicong Lin}
\address[Z. Lin]{Research Center for Mathematics and Interdisciplinary Sciences, Shandong University, Qingdao 266237, P.R. China}
\email{linz@sdu.edu.cn}

\date{}

\begin{document}

\begin{abstract}
	
	Recently, Hong and Li launched a systematic study of length-four pattern avoidance in inversion sequences, and in particular, they conjectured that the number of $0021$-avoiding inversion sequences can be enumerated by the OEIS entry A218225. Meanwhile, Burstein suggested that the same sequence might also count three sets of pattern restricted permutations. The objective of this paper is not only a confirmation of Hong and Li's conjecture and Burstein's first conjecture, but also two more delicate generating function identities with the $\mathsf{ides}$ statistic concerned in the restricted permutation case, and the $\mathsf{asc}$ statistic concerned in the restricted inversion sequence case, which yield a new equidistribution result.


	\MSC{05A05, 05A15.}
	
\end{abstract}

\maketitle

\section{Introduction}

The history of the study of patterns in permutations dates back to the first volume of MacMahon's 1915 magnum opus \textit{Combinatory Analysis}~\cite[Vol.~I, Sect.~III, Ch.~V]{Mac1915}. Meanwhile, modern treatment of permutation classes is commonly known to have its first appearance in Knuth's Volume 1 of~\textit{The Art of Computer Programming}~\cite[Sect.~2.2.1]{Knu1997}, another masterpiece in discrete mathematics and computer science. See Kitaev's monograph~\cite{Kit2011} and Vatter's survey~\cite{Vat2015} for detailed accounts of this charming history.

Let $\fS_n$ be the set of permutations of $\{1,2,\ldots,n\}=:[n]$. We know that it has a natural coding using $\bI_n:=\{(e_1,e_2,\ldots,e_n):0\le e_i<i\}$ given by
\begin{align}\label{eq:natural-coding}
	\Theta:\pi=(\pi_1,\pi_2,\ldots,\pi_n)\mapsto (e_1,e_2,\ldots,e_n),\quad\text{where $e_i=|\{j<i:\pi_j>\pi_i\}|$}.
\end{align}
Since the sum of the entries in $\Theta(\pi)$ gives the number of \textit{inversions} (i.e., pairs $(i,j)$ with $i<j$ and $\pi_i>\pi_j$) of $\pi$, we usually call sequences in $\bI_n$ \textit{inversion sequences} of length $n$.

Given a sequence $W=w_1w_2\cdots w_n$, we say that $W$ \textit{contains} a fixed pattern $P=p_1p_2\cdots p_k$ if there is a subsequence of $W$ that is order isomorphic to $P$. Otherwise, we say that $W$ \textit{avoids} the pattern $P$. For example, the sequence $w_1w_2\cdots w_6=315616$ contains the pattern $011$ since the subsequence $w_1w_4w_6=366$ is order isomorphic to $011$, but avoids $201$ since none of the subsequences are order isomorphic to this pattern.

In general, for $P_1,\ldots,P_r$, a family of patterns, we denote by $\fS_n(P_1,\ldots,P_r)$ (resp.~$\bI_n(P_1,\ldots,P_r)$) the set of permutations (resp.~inversion sequences) that simultaneously avoid $P_1,\ldots,P_r$.

One of the most important questions in the prolific study of patterns in permutations or inversion sequences concerns the enumeration of sequences avoiding a certain family of patterns. As a well-known example \cite[Proposition 2.1.3]{Kit2011}, one has $|\fS_n(231)|=\frac{1}{n+1}\binom{2n}{n}$, the $n$-th Catalan number. This result is partly attributed to Knuth \cite[pp.~242--243]{Knu1997}. For pattern avoidance in inversion sequences, on the other hand, we witness the pioneering works of Corteel \textit{et al.}~\cite{CMSW2016} and Mansour--Shattuck~\cite{MS2015}. For instance, $|\bI_n(012)|$ equals the $(2n-1)$-th Fibonacci number, while $|\bI_n(000)|$ is the $(n+1)$-th Euler  number. Further intriguing  connections between pattern avoiding permutations and pattern avoiding inversion sequences were investigated extensively in~\cite{LK2018,LK2021,Lin2020,FL2021}. 

It is worth noting that for permutations, when the pattern family contains patterns of different lengths, their enumerations enjoy wide-ranging motivations and may lead to unexpected links. For instance, as a byproduct from their combinatorial study of Schubert polynomials, Billey, Jockusch, and Stanley \cite{BJS1993} showed that
$$|\fS_n(321,2143)|=2^{n+1}-\binom{n+1}{3}-2n-1.$$
A result involving patterns of lengths $4$ and $6$, which was first conjectured by Egge, and later confirmed by Burstein and Pantone \cite{BP2015}, Bloom and Burstein \cite{BB2016}, states that for each $\tau\in\{246135,254613,\allowbreak 524361,546132,263514\}$, we have
$$|\fS_n(2143,3142,\tau)|=\sum_{d=0}^n\frac{1}{n-d+1}\binom{2n-d}{n-d,n-d,d},$$
the $n$-th large Schr\"oder number. Quite recently, during his systematic investigation of West's ``stack-sorting map'', Defant \cite{Def2021} enumerated $\fS_n(2341,3241,45231)$ by recognizing it as the preimages set of $\fS_n(231,321)$.

As the study of such enumerations continues, many open problems also emerge in an endless stream. Let
$$\cA(x)=\sum_{n\ge 1}a(n)x^n:=x+2x^2+6x^3+23x^4+101x^5+480x^6+\cdots$$
be the unique formal power series solution to the  functional equation
\begin{align}\label{eq:A-eq}
	\cA(x)=\big(1+\cA(x)\big)\big(x+\cA(x)^2-x\cA(x)^2\big).
\end{align}
By the Lagrange inversion formula, we have 
\begin{align*}
a(n)&=\frac{1}{n}[x^{n-1}]\biggl(\frac{1}{1-x-x^2}+x\biggr)^n\\
&=\frac{1}{n}\sum_{k=0}^{n-1}\binom{n}{k+1}[x^k](1-x-x^2)^{-(k+1)}\\
&=\frac{1}{n}\sum_{k=0}^{n-1}\binom{n}{k+1}\sum_{j=0}^k\binom{k+j}{j}\binom{j}{k-j}.
\end{align*}
The coefficients of $\cA(x)$ are registered as sequence A218225\footnote{This OEIS sequence starts with the zeroth entry so its generating function is $\cA(x)/x$ in our notation.} in the OEIS~\cite{OEIS}. On December 20\textsuperscript{th}, 2017, Burstein posed in~\cite[A218225]{OEIS} the following conjecture.

\begin{conjecture}[Burstein]\label{conj:Bur}
	The value $a(n)$ counts the number of permutations of length $n$ that avoid one of the following sets of patterns:
	$$(3124, 42153, 24153),\quad (2134, 42153, 24153),\quad (2143, 42135, 24135).$$
\end{conjecture}

On the other hand, following Chern's work \cite{Che2022} on 0012-avoiding inversion sequences, Hong and Li \cite{HL2021} recently presented a systematic study of other length-four pattern restricted cases. In particular, they conjectured that the series $\cA(x)$ also generates $|\bI_n(0021)|$; see \cite[Conjecture 4.3]{HL2021}.

\begin{conjecture}[Hong--Li]\label{conj:HL}
	The value $a(n)$ counts the number of inversion sequences of length $n$ that avoid the pattern $0021$.
\end{conjecture}

Yet another topic of substantial significance is the investigation of various statistics for permutations or inversion sequences. Here, we list a few of them that are commonly used: for $w$ a permutation or an inversion sequence of length $n$,
\begin{itemize}[leftmargin=*,align=left]
	\renewcommand{\labelitemi}{$\triangleright$}
	\item \textit{the number of \textbf{asc}ents}: $\asc(w):=|\{i\in[n-1]:w_i<w_{i+1}\}|$;
	
	\item \textit{the number of \textbf{des}cents}: $\des(w):=|\{i\in[n-1]:w_i>w_{i+1}\}|$;
\end{itemize}
for $\pi$ a permutation of length $n$,
\begin{itemize}[leftmargin=*,align=left]
	\renewcommand{\labelitemi}{$\triangleright$}
	\item \textit{the number of \textbf{i}nverse \textbf{des}cents}: $\ides(\pi):=\des(\pi^{-1})$ where $\pi^{-1}$ is the inverse permutation of $\pi$;
	
	\item \textit{the number of \textbf{exc}edances}: $\exc(\pi):=|\{i\in[n]:\pi_i>i\}|$; 
\end{itemize}
for $e$ an inversion sequence of length $n$,
\begin{itemize}[leftmargin=*,align=left]
	\renewcommand{\labelitemi}{$\triangleright$}
	\item \textit{the number of \textbf{dist}inct positive entries}: $\dist(e):=|\{e_i:\text{$i\in[n]$ and $e_i>0$}\}|$.
\end{itemize}

Recall that for $n\ge 1, k\ge 0$, the \textit{Eulerian numbers} are defined by
$$\left \langle \begin{matrix}
	n\\k
\end{matrix} \right \rangle :=\sum_{j=0}^{k+1} (-1)^j \binom{n+1}{j} (k-j+1)^n.$$
We may also define the \textit{Eulerian polynomials} by
$$E_n(t):=\sum_{k=0}^{n-1}\left \langle \begin{matrix}
	n\\k
\end{matrix} \right \rangle t^k.$$
It is known (see, e.g.~\cite[Sect.~1.4]{Per2015}) that
$$\sum_{\pi\in\fS_n}t^{\des(\pi)} = E_n(t).$$
Usually, we say that a statistic is \textit{Eulerian} if its distribution gives $E_n(t)$. Thus, $\des$, and equivalently, $\asc$ and $\ides$ are Eulerian over permutations. Also, a direct implication of the natural coding $\Theta$ in \eqref{eq:natural-coding} is that for any $\pi\in\fS_n$, $\des(\pi)=\asc(\Theta(\pi))$. Hence, $\asc$ is also Eulerian over inversion sequences. For some nontrivial examples, Foata and Sch\"utzenberger \cite{FS1970} showed by a bijection called the ``transformation fondamentale'' from $\fS_n$ to itself that $\exc$ is Eulerian over permutations, and Dumont~\cite{Dum1974} proved that $\dist$ is Eulerian over inversion sequences.

In general, two statistics that are equidistributed over permutations and/or inversion sequences, are a priori no longer equidistributed when the two sets are restricted by certain pattern avoidance conditions, even if the two restricted subsets are already known to be equinumerous. However, there are still a few exceptions. For example, by the works of Corteel \textit{et al.}~\cite{CMSW2016} and Fu \textit{et al.}~\cite{FLZ2018},
\begin{align}\label{eq:C-F}
	\sum_{\pi\in\fS_n(2413,3142)}t^{\des(\pi)} = \sum_{e\in\bI_n(021)} t^{\asc(e)}.
\end{align}
Inspired by the above relation,  Lin and Kim~\cite{LK2018} constructed a bijection between $\fS_n(2413,4213)$ and $\bI_n(021)$, which proves a surprising sextuple equidistribution, including the set-valued statistics of the positions of  descents and ascents.

In this paper, we will focus our attention on Burstein's first conjecture and Hong and Li's conjecture, but also with statistics considered. To begin with, we generalize the series $\cA(x)$ as
\begin{align}\label{eq:A(x,t)-def}
	\cA(x,t):=x + (1 + t) x^2 + (1 + 4 t + t^2) x^3 + (1 + 10 t + 11 t^2 + t^3) x^4 + \cdots,
\end{align}
which is the unique formal power series solution to the functional equation
\begin{align}\label{eq:A(x,t)-eqn}
	\cA(x,t)=\big(1+\cA(x,t)\big)\big(x+t \cA(x,t)^2 - xt^2 \cA(x,t)^2\big).
\end{align}
Note that $\cA(x,1)=\cA(x)$. %In what follows, by abuse of notation, we will sometimes write $\cA(x,t)$ as $\cA(x)$ for notational convenience.

The main result of this paper is the following theorem, which affords us two combinatorial interpretations of $\cA(x,t)$ in terms of the Eulerian distributions over pattern restricted permutations and inversion sequences, respectively.

\begin{theorem}
	Let $\cA(x,t)$ be as in \eqref{eq:A(x,t)-def}. Then
	\begin{align}\label{eq:P-gf}
		\sum_{n\ge 1}x^n\sum_{\pi\in\fS_n(3124, 42153, 24153)}t^{\ides(\pi)} = \cA(x,t)
	\end{align}
	and
	\begin{align}\label{eq:I-gf}
		\sum_{n\ge 1}x^n\sum_{e\in\bI_n(0021)}t^{\asc(e)} = \cA(x,t).
	\end{align}
	Hence, $\ides$ over $\fS_n(3124, 42153, 24153)$ is equidistributed with $\asc$ over $\bI_n(0021)$.
\end{theorem}

The following is a quick consequence of our main result by setting $t=1$.

\begin{corollary}
	\textbf{\emph{(i).}}~Burstein's First Conjecture is true, i.e.~Conjecture \ref{conj:Bur} is true for $(3124, 42153, 24153)$; \textbf{\emph{(ii).}}~Hong and Li's Conjecture \ref{conj:HL} is true.
\end{corollary}

\section{Burstein's first conjecture with inverse descents}

For convenience, we write $I:=\{3124,42153,24153\}$, the set of patterns in Burstein's first conjecture. Let us define
\begin{align}
	\cP(x,t):=\sum_{n\ge 1}x^n\sum_{\pi\in \fS_{n}(I)}t^{\ides(\pi)}.
\end{align}
We begin with some useful definitions and initial observations.

We associate with each permutation $\pi\in\fS_{n+1}$ a unique word $w_{\pi}$ of length $n$ consisting of letters $\mathrm{L}$ and $\mathrm{R}$ as follows. Suppose $\pi_k=1$, we let
$$
w_{\pi}=w_1w_2\cdots w_n, \text{ where }w_i:=\begin{cases}
	\mathrm{L}, & \text{if $\pi_j=i+1$ for a certain $j<k$,}\\
	\mathrm{R}, & \text{if $\pi_j=i+1$ for a certain $j>k$.}
\end{cases}
$$
Intuitively, as we scan the entries of $\pi$ from $2$ to $n+1$, the word $w_{\pi}$ simply records whether we are at a position that is to the left or right of the entry $1$. We introduce  the length of the longest alternating subword of any word $w$ on the alphabet $\{\mathrm{L},\mathrm{R}\}$:
$$\alt(w):=\max\{|u|:\text{$u$ is a subword of $w$ with no consecutively repeated letters}\}.$$
For a permutation $\pi\in\fS_n$, we define
\begin{itemize}[leftmargin=*,align=left]
	\renewcommand{\labelitemi}{$\triangleright$}
	\item \textit{the \textbf{alt}ernating length}: $\alt(\pi):=0$ for $n=1$, and $\alt(\pi):=\alt(w_{\pi})$ for $n\ge 2$.
\end{itemize}

Denote $\fS_{n,k}$ the set of permutations in $\fS_n$ with alternating length equal to $k$. 

\begin{example}
	The permutation $\pi=547912683$ corresponds to the word $w_{\pi}=\mathrm{R}\mathrm{R}\mathrm{L}\mathrm{L}\mathrm{R}\mathrm{L}\mathrm{R}\mathrm{L}$ (for $2,3$ are to the \textsl{R}ight of $1$; $4$ is to the \textsl{L}eft of $1$; etc.). Thus, $\alt(\pi)=6$ and $\pi\in\fS_{9,6}$.
\end{example}

The next lemma explains why it might be a good idea to refine $\fS_n$ as $\fS_{n,k}$ when we deal with $\fS_n(I)$.

\begin{lemma}\label{lem:alt-bound}
	For each permutation $\pi\in\fS_n(I)$, there is no subword of $w_\pi$ of the form $\mathrm{L}\mathrm{R}\mathrm{L}\mathrm{R}$. Consequenctly, we have $\alt(\pi)\le 4$, and in particular, if $w_\pi$ starts with $\mathrm{L}$ (i.e.~$\pi^{-1}_2<\pi^{-1}_1$), we have $\alt(\pi)\le 3$.
\end{lemma}

\begin{proof}
	Taking any permutation $\pi\in\fS_n(I)$ with $\pi_k=1$, suppose on the contrary that $\mathrm{L}\mathrm{R}\mathrm{L}\mathrm{R}$ is a subword of $w_\pi$. Then it is always possible to find four indices $a,b,c,d$, such that
	\begin{enumerate}[label=(\arabic*), widest=9, itemindent=*, leftmargin=*]
		\item $a<k$, $b<k$ and $c>k$, $d>k$;
		\item $\pi_a<\pi_c<\pi_b<\pi_d$.
	\end{enumerate} 
	Furthermore, since $\pi$ avoids the pattern $3124$, we must have $d<c$. But then if $a<b$, we see $$\pi_a,\pi_b,\pi_k=1,\pi_d,\pi_c$$ matches the pattern $24153$; while if $a>b$, we get from $$\pi_b,\pi_a,\pi_k=1,\pi_d,\pi_c$$ the pattern $42153$. Either case contradicts with the fact that $\pi\in\fS_n(I)$.
\end{proof}

Throughout, we make the following decomposition
\begin{align*}
	\fS_n(I) &= \biguplus_{i=0}^4 \fS_{n,i}(I).
\end{align*}
Note that if $\alt(\pi)=t$, then the longest alternating subword of $w_{\pi}$ could be either $$u=\underbrace{\mathrm{L}\mathrm{R}\mathrm{L}\ldots}_{t}\quad\text{or}\quad u=\underbrace{\mathrm{R}\mathrm{L}\mathrm{R}\ldots}_{t}.$$
Fortunately, according to the next lemma, it suffices to investigate the first case. For $k\ge 1$, we decompose $\fS_{n,k}$ further into two disjoint subsets (with $\pi_i^{-1}$ the $i$-th entry of $\pi^{-1}$):
\begin{align*}
	\fS_{n,k}^{L}&:=\{\pi\in\fS_{n,k}: \pi^{-1}_2<\pi^{-1}_1\},\\
	\fS_{n,k}^{R}&:=\{\pi\in\fS_{n,k}: \pi^{-1}_2>\pi^{-1}_1\}.
\end{align*}
We then denote the corresponding generating functions for those permutations avoiding the pattern set $I$ as
\begin{align*}
	\cP_k^L(x,t) &:=\sum_{n\ge 2}x^n\sum_{\pi\in \fS_{n,k}^L(I)}t^{\ides(\pi)},\\
	\cP_k^R(x,t) &:=\sum_{n\ge 2}x^n\sum_{\pi\in \fS_{n,k}^R(I)}t^{\ides(\pi)},
\end{align*}
respectively.

\begin{lemma}\label{lem:LandR}
	We have
	\begin{align}\label{id:F^R}
		\cP_1^R(x,t)=x\cP(x,t),
	\end{align} 
	and for $1\le k\le 3$, 
	\begin{align}\label{id:LandR}
		\cP_{k+1}^R(x,t)=\cP_{k}^L(x,t)\cP(x,t).
	\end{align}
\end{lemma}

\begin{proof}
	Given any permutation $\pi=\pi_1\cdots\pi_n\in\fS_n(I)$, we see that $\hat{\pi}=1(\pi_1+1)\cdots(\pi_n+1)$ is a permutation in $\fS_{n+1,1}^R(I)$. Conversely, every permutation in $\fS_{n+1,1}^R(I)$ must begin with $1$ and recovers uniquely a permutation in $\fS_n(I)$ once the $1$ is removed and all remaining entries decrease by $1$. Noting that this correspondence also preserves the $\ides$ statistic, we arrive at \eqref{id:F^R}.
	
	Similarly, for $1\le k\le 3$, we can construct a bijection 
	\begin{align*}
		\begin{array}{cccc}
			\phi: & \bigcup_{n\ge 2}\fS_{n,k}^L(I)\times\bigcup_{n\ge 1}\fS_{n}(I) & \to & \bigcup_{n\ge 2}\fS_{n,k+1}^R(I),\\[6pt]
			& (\sigma,\mu) & \mapsto & \pi,
		\end{array}
	\end{align*}
	where the image $\pi$ is obtained from the pair $(\sigma,\mu)$ by first concatenating $\sigma$ with $\mu$, then increasing each entry of $\sigma$ larger than $1$ by $|\mu|$ and each entry of $\mu$ by $1$. For instance, 
	$$\phi(6271435,3142)=10\:6\:11\:1\:8\:7\:9\:4\:2\:5\:3.$$
	A moment of reflection should reveal that regardless of $k=1$, $2$ or $3$, $\phi$ is well-defined and indeed a bijection satisfying that $|\phi(\sigma,\mu)|=|\sigma|+|\mu|$. In terms of generating function, we see that \eqref{id:LandR} holds by further noting that $\ides(\pi)=\ides(\sigma)+\ides(\mu)$, since the eliminated inverse descent from $\sigma_1^{-1}>\sigma_2^{-1}$ is compensated by $\pi_{|\mu|+1}^{-1}>\pi_{|\mu|+2}^{-1}$.
\end{proof}

Noting from Lemma \ref{lem:alt-bound} that $\fS_{n,4}^L(I)=\varnothing$ for all $n\ge 1$, we have
\begin{equation}
	\cP_{4}^L(x,t)=0.
\end{equation}
Now, it remains to compute $\cP_1^L(x,t)$, $\cP_2^L(x,t)$ and $\cP_3^L(x,t)$. We collect them in the next lemma.

\begin{lemma}\label{lem:F123}
	We have
	\begin{align}
		\cP_1^L(x,t) &= xt\cP(x,t),\label{F1}\\
		\cP_2^L(x,t) &= xt\cC(x,t)\cP(x,t),\label{F2}\\
		\cP_3^L(x,t) &= t\big(\cP(x,t)-x-xt\cP(x,t)-x\cC(x,t)\big)\cP(x,t),\label{F3}
	\end{align}
	where
	$$\cC(x,t)=\sum_{n\ge 1}x^n\sum_{\pi\in \fS_n(312)}t^{\ides(\pi)}.$$
\end{lemma}

\begin{proof}
	First, \eqref{F1} can be derived bijectively in the same way as we prove \eqref{id:F^R}, except that now we append $1$ at the right end, thereby increasing the $\ides$ statistic by $1$.
	
	For \eqref{F2}, suppose that $\pi\in\fS_{n,2}^L(I)$ decomposes as $\pi=u 1 v$ with $\max(u)<\min(v)$. Now as two subwords of $\pi$, $u$ and $v$ each should avoid the three patterns contained in $I$. More precisely, we see that $u$ actually avoids the pattern $312$, since any occurrence of $312$ in $u$ together with one entry from $v$ outputs the pattern $3124$. Conversely, by concatenating a permutation $\sigma\in\fS_m(312)$ with $1$ and a permutation $\mu\in\fS_l(I)$, then increasing every entry in $\sigma$ by $1$ and every entry in $\mu$ by $m+1$, we get a permutation $\pi\in\fS_{m+l+1,2}^L(I)$. This bijection gives us \eqref{F2} by noting that $\ides(\pi)=\ides(\sigma)+\ides(\mu)+1$ for the inverse descent from $\pi_1^{-1}>\pi_2^{-1}$ gives the additional ``plus one.''
	
	The proof of \eqref{F3} requires a bijection that is less obvious. We begin with any permutation $\sigma=\sigma_1\cdots\sigma_n\in\fS_n(I)$. Increasing all entries by $1$ and inserting $1$ at the penultimate position yield
	$$\tilde{\sigma}=(\sigma_1+1)\cdots(\sigma_{n-1}+1)1(\sigma_n+1).$$
	Let us decide under what conditions $\tilde{\sigma}$ would be a permutation in $\fS_{n+1,3}^L(I)$.
	\begin{enumerate}[label=(\arabic*), widest=9, itemindent=*, leftmargin=*]
		\item $n\ge 2$. This explains the term ``$-x$.''
		\item $\sigma_n>1$. Otherwise, $\sigma_n=1$ and $\sigma\in\fS_{n,1}^L(I)$. So by \eqref{F1}, this explains the term ``$-xt\cP(x,t)$.''
		\item $\sigma_n<n$. Otherwise, $\sigma_n=n$ and $\sigma_1\cdots\sigma_{n-1}\in\fS_{n-1}(312)$. So this explains the term ``$-x\cC(x,t)$.''
	\end{enumerate}
	If conditions (1), (2) and (3) are all satisfied, we see that $\tilde{\sigma}$ is indeed a permutation in $\fS_{n+1,3}^L(I)$ with $1$ being at the penultimate position. To reach all permutations in $\bigcup_{i\ge 4}\fS_{i,3}^L(I)$, we simply ``inflate'' the last entry of $\tilde{\sigma}$ into any permutation $\mu\in\fS(I)$ (i.e., increase each entry in $\mu$ by $\sigma_n$ to get $\tilde{\mu}$, and then replace the last entry of $\tilde{\sigma}$ with $\tilde{\mu}$), and then adjust the remaining entries of $\tilde{\sigma}$ if necessary (i.e., increase by an appropriate amount and preserve their original relative order relations). Denote the output permutation as $\pi$. Now, this final step of ``inflation'' corresponds to multiplying by $t\cP(x,t)$, by further noting that $\ides(\pi)=\ides(\tilde{\sigma})+\ides(\mu)=1+\ides(\sigma)+\ides(\mu)$. For example, if
	$$\sigma=4213 \quad\text{and}\quad \mu=41523,$$
	we follow the above steps to produce
	$$\tilde{\sigma}=532\:1\:4 \quad\text{and}\quad  \tilde{\mu}=74856 \quad\text{and}\quad \pi=932\:1\:74856.$$
	It is not hard to see that the whole process is invertible. Namely, we take any permutation from $\bigcup_{n\ge 4}\fS_{n,3}^L(I)$, ``concentrate'' all the entries to the right of $1$ into its minimum to recover $\tilde{\sigma}$ after standardization, and then remove $1$ and decrease other entries by $1$ to further recover $\sigma$, which must be a permutation in $\fS_n(I)$ that satisfies conditions (1), (2) and (3). In summary, $(\sigma,\mu)\mapsto \pi$ is a bijection and \eqref{F3} is now proved.
\end{proof}

Finally, we are in a position to conclude the proof of \eqref{eq:P-gf}.

\begin{proof}[Proof of \eqref{eq:P-gf}]
	Putting together these generating function relations, we can deduce the following functional equation:
	\begin{align*}
		\cP(x,t) &= \big(x+\cP_1^R(x,t)\big)+\big(\cP_1^L(x,t)+\cP_2^R(x,t)\big)+\big(\cP_2^L(x,t)+\cP_3^R(x,t)\big)\\
		&\quad+\big(\cP_3^L(x,t)+\cP_4^R(x,t)\big)+\cP_4^L(x,t)\\
		&= \big(1+\cP(x,t)\big)\big[x+xt \cP(x,t)+xt \cC(x,t)\cP(x,t)\\
		&\quad+t\big(\cP(x,t)-x-xt\cP(x,t)-x\cC(x,t)\big)\cP(x,t)\big]\\
		&= \big(1+\cP(x,t)\big)\big(x+t\cP(x,t)^2-xt^2\cP(x,t)^2\big),
	\end{align*}
	which simplifies to an identical equation as \eqref{eq:A(x,t)-eqn}. Hence, we arrive at \eqref{eq:P-gf}.
\end{proof}

\section{Hong and Li's conjecture}

For any inversion sequence $e\in\bI_n$, it is always possible to find a unique index $p$ such that $e_i=i-1$ for $1\le i\le p$ and $e_{p+1}\le p-1$. Note that here $p$ may take the value $n$, in which case, we shall view the entry $e_{n+1}$, which does not exist, as $-\infty$. We define
\begin{itemize}[leftmargin=*,align=left]
	\renewcommand{\labelitemi}{$\triangleright$}
	\item \textit{the \textbf{i}nitial \textbf{a}scending \textbf{r}un}: $\iar(e):=p$, the aforementioned index.
\end{itemize}

Now, we consider a generalization of  inversion sequences introduced by Savage and  Schuster~\cite{SS2012}.

\begin{definition}
	Let $\bs=(s_1,s_2,\ldots,s_n)$ be a fixed sequence of positive integers. A sequence $e=(e_1,e_2,\ldots,e_n)$ is an \textit{$\bs$-inversion sequence} if $0\le e_i<s_i$ for all $1\le i\le n$. Further, an entry $e_i$ is \textit{tight} if $e_i=s_i-1$.
\end{definition}

For $e$ an $\bs$-inversion sequence of length $n$, we define
\begin{itemize}[leftmargin=*,align=left]
	\renewcommand{\labelitemi}{$\triangleright$}
	\item \textit{the number of \textbf{tig}ht entries}: $\tig(e):=|\{i\in[n]:e_i=s_i-1\}|$.
\end{itemize}

From the concept of the initial ascending runs, we are naturally led to the set of $(p,p+2,p+3,\ldots,p+n)$-inversion sequences, denoted by $\bI_{n,p}$, for $n$ and $p$ positive integers. The following lemma is easy but crucial. 

\begin{lemma}\label{lem:0021}
	Fix positive integers $p$ and $n$. There is a natural one-to-one correspondence between 
	$$
	\{e\in\bI_{n+p}(0021):\iar(e)=p\}\quad\text{and}\quad \bI_{n,p}(021). 
	$$
\end{lemma}
\begin{proof}
	As every sequence $e\in\bI_{n+p}(0021)$ with $\iar(e)=p$ is of the form $(0,1,2,\ldots,p-1,e_{p+1},e_{p+2},\ldots,e_{p+n})$ such that $0\leq e_{p+1}\leq p-1$, deleting the first $p$ entries of $e$ produces the sequence $\hat{e}=(e_{p+1},e_{p+2},\ldots,e_{p+n})\in\bI_{n,p}$. It is routine to check that $\hat{e}$ is a sequence in $\bI_{n,p}(021)$ and this correspondence is one-to-one. 
\end{proof}

Let
\begin{align}
	 \cI(x):=\sum_{n\geq1}|\bI_n(0021)|x^n.
\end{align}
The objective of this section is a short proof of Hong and Li's conjecture, without considering the $\asc$ statistic. To begin with, we introduce the generating function
$$
\cH(x,s,q):=\sum_{n,p\geq1}x^ns^p\sum_{e\in\bI_{n,p}(021)}q^{\tig(e)}.
$$
We also write the coefficients as
\begin{align*}
	h_p(x,q)&:=[s^p]\cH(x,s,q),\\
	h_{p,k}(x)&:=[s^pq^k]\cH(x,s,q).
\end{align*}

\begin{proof}[First Proof of Hong and Li's Conjecture]
	Let $e=(e_1,\ldots,e_n)\in\bI_{n,p}(021)$. We decompose $e$ by considering its right-most tight entry.  
	There are two cases:
	
	\textbf{(1).}~\textit{$e$ has no tight entry.} Such sequences are in one-to-one correspondence with the sequences in $\bI_{n,p-1}(021)$ and thus contribute the generating function
	$$
	\sum_{p\geq2}h_{p-1}(x,1)s^p=s\cH(x,s,1).
	$$
	
	\textbf{(2).}~\textit{$e$ has at least one tight entry.} Suppose that  $e_{\ell}$ is the right-most tight entry of $e$. There are four subcases:
	
	\textbf{(2a).}~\textit{$\ell=n$.} Every such sequence can be obtained from a (possibly empty) sequence in $\bI_{n-1,p}(021)$ by adding a tight entry at the end. Thus, this case contributes the generating function
	$$
	\sum_{p\geq1}xq\big(1+h_p(x,q)\big)s^p=xq\bigg(\frac{s}{1-s}+\cH(x,s,q)\bigg).
	$$ 
	
	\textbf{(2b).}~\textit{$1\leq\ell<n$ and $e_{\ell+1}=e_{\ell}$.} Each sequence in this case can be obtained from a sequence in $\bI_{n-1,p}(021)$ with at least one tight entry  by inserting exactly to the right of any tight entry a new entry of equal size. Note that this construction is in general one-to-multiple, and if it is executed at the $j$-th tight entry, then the resulting sequence has precisely $j$ tight entries. Therefore,  this case contributes the generating function
	\begin{align*}
		x\sum_{p,k\geq1}h_{p,k}(x)\big(q+q^2+\cdots+q^k\big)s^p&=x\sum_{p\geq1,k\geq0}h_{p,k}(x)\bigg(\frac{q-q^{k+1}}{1-q}\bigg)s^p\\
		&=\frac{xq}{1-q}\big(\cH(x,s,1)-\cH(x,s,q)\big).
	\end{align*}

	\textbf{(2c).}~\textit{$\ell=1<n$ and either $e_{\ell+1}<e_{\ell}$ or $e_{\ell+1}=e_{\ell}+1$.} In this case $e_1=p-1$, and in the latter situation $e_2=p$. Also, the $\tig$ statistic equals $1$. Such sequences are in one-to-one correspondence with $\bI_{n-1,p}(021)$ by first deleting $e_1$, and if $e_2=p$, changing this entry to $p-1$.   It follows that the  sequences in this case contribute the generating function
	$$
	xq\sum_{p\geq1}h_{p}(x,1)s^p=xq\cH(x,s,1).
	$$
	
	\textbf{(2d).}~\textit{$1<\ell<n$ and $e_{\ell+1}<e_{\ell}$.} In this case, if we denote  $\min\{e_i: 1\leq i<\ell\}$ by $m$, then $e$ can be decomposed into two shorter sequences
	$$
	\tilde{e}:=(e_1,e_2,\ldots,e_{\ell-1})\in\bI_{\ell-1,p,m}(021)
	$$
	and 
	$$\bar{e}:=(\bar e_{\ell+1},\bar e_{\ell+2},\ldots,\bar{e}_n)\in\bI_{n-\ell,m+1}(021),$$
	where $\bI_{\ell-1,p,m}(021):=\{e\in\bI_{\ell-1,p}(021):\text{ the minimal entry of $e$ is $m$}\}$ and 
	$$
	\bar e_i:=
	\begin{cases}
		e_i&\quad\text{if $e_i<e_{\ell}$, (one has $e_i\leq m$ since $e$ avoids $021$)}\\
		e_i-e_{\ell}+m+1&\quad\text{if $e_i\geq e_{\ell}$,}
	\end{cases}
	$$
	for $\ell+1\leq i\leq n$. This decomposition is reversible. The key observation is that for fixed $p\geq1$ and $0\leq m\leq p-1$, a member in $\bI_{\ell,p,m}(021)$ can be obtained from a member $e\in\bI_{\ell,p-m}(021)$ provided that $e$ has $0$ as its minimal entry. 
	It follows that the generating function 
	$\sum_{\ell\geq1}x^{\ell}\sum_{e\in\bI_{\ell,p,m}(021)}q^{\tig(e)}$ equals $h_{p-m}(x,q)-h_{p-m-1}(x,q)$, where by convention $h_0(x,q)=0$. Therefore, this case contributes the generating function 
	\begin{align*}
		&\sum_{p\geq1}xqs^p\sum_{m=0}^{p-1}\big(h_{p-m}(x,q)-h_{p-m-1}(x,q)\big) h_{m+1}(x,1)\\
		&\quad=\frac{xq}{s}\sum_{p\geq1}\sum_{m=0}^{p-1}h_{p-m}(x,q)h_{m+1}(x,1)s^{p+1}-xq\sum_{p\geq1}\sum_{m=0}^{p-1}h_{p-m-1}(x,q)h_{m+1}(x,1)s^p\\
		&\quad=xq\big(1/s-1\big)\cH(x,s,q)\cH(x,s,1). 
	\end{align*}

	Summing over all the above cases yields the functional equation for $\cH(q):=\cH(x,s,q)$:
	\begin{equation}\label{eq:C}
		\bigg(\frac{1-q+xq^2}{1-q}-xq\big(1/s-1\big)\cH(1)\bigg)\cH(q)=\bigg(s+xq+\frac{xq}{1-q}\bigg)\cH(1)+\frac{xqs}{1-s}.
	\end{equation}
	We apply the kernel method to solve \eqref{eq:C} by setting the kernel polynomial
	$$\frac{1-q+xq^2}{1-q}-xq\big(1/s-1\big)\cH(1)$$ 
	to be zero. Then the right-hand side of \eqref{eq:C} vanishes. Therefore, $\cH(1)$ satisfies the system of equations
	\begin{equation}\label{sys:ker1}
		\begin{cases}
			\dfrac{1-q+xq^2}{1-q}-xq\big(1/s-1\big)\cH(1)=0,\\[12pt]
			\bigg(s+xq+\dfrac{xq}{1-q}\bigg)\cH(1)+\dfrac{xqs}{1-s}=0.
		\end{cases}
	\end{equation}

	Finally, we recall that any inversion sequence in $\bI_n(0021)$ is either in correspondence with a sequence in $\bI_{n-p,p}(021)$ for some $p$ by Lemma \ref{lem:0021}, or a sequence with only tight entries, which is generated by
	$$x+ x^2+ x^3+\cdots = \frac{x}{1-x}.$$
	Hence,
	$$
	\cI(x)=\frac{x}{1-x}+\cH(x,x,1).
	$$
	Equivalently, 
	$$
	\cH(x,x,1)=\cI(x)-\frac{x}{1-x}. 
	$$
	Substituting this expression into \eqref{sys:ker1} gives the functional equation for $\cI(x)$:
	\begin{equation*}
		\begin{cases}
			\dfrac{1-q+xq^2}{1-q}-(1-x)q\bigg(\cI(x)-\dfrac{x}{1-x}\bigg)=0,\\[12pt]
			\dfrac{x(1+q-q^2)}{1-q}\bigg(\cI(x)-\dfrac{x}{1-x}\bigg)+\dfrac{x^2q}{1-x}=0.
		\end{cases}
	\end{equation*}
	Canceling $q$ in this system gives
	$$\cI(x)=\big(1+\cI(x)\big)\big(x+\cI(x)^2-x\cI(x)^2\big),$$
	which is identical to \eqref{eq:A-eq}.
\end{proof}

\begin{remark}
	The above analysis does not work well if the $\asc$ statistic is taken into account. The main trouble occurs in Case (2c). Recall that in this case, for $e=(e_1,\ldots,e_n)\in\bI_{n,p}(021)$, we have $e_1=p-1$ and either $e_2<p-1$ or $e_2=p$. When $e_2<p-1$, the $\asc$ statistic remains the same after applying the correspondence. However, when $e_2=p$, we find that in the resulting sequence, the $\asc$ statistic decreases by $1$ if $e_3\le p-1$ or $e_3=p+1$; and remains the same value if $e_3=p$. Such a bifurcating behavior of the $\asc$ statistic keeps us away from a neat generating function for this case.
\end{remark}

\section{General \texorpdfstring{$021$}{}-avoiding sequences}

To attach the $\asc$ statistic to the generating function $\cI(x)=\sum_{n\ge 1}|\bI_n(0021)|x^n$, we have to undertake a more subtle analysis for $\bI_{n,p}$. This propels us to look at general $021$-avoiding sequences. Let $\bN$ denote the set of sequences of nonnegative integers, and let $\bN_n$ denote the set of sequences of length $n$ in $\bN$ for $n\ge 1$. For a sequence $w=w_1w_2\cdots w_n\in\bN_n$, we define
\begin{itemize}[leftmargin=*,align=left]
	\renewcommand{\labelitemi}{$\triangleright$}
	\item \textit{the \textbf{lar}gest entry}: $\lar(w):=\max\{w_i:1\le i\le n\}$;
	
	\item \textit{the \textbf{sma}llest entry}: $\sma(w):=\min\{w_i:1\le i\le n\}$;
	
	\item \textit{the \textbf{r}ight-most position of the \textbf{lar}gest entry}: $\RLAR(w):=\max\{i:w_i=\lar(w)\}$;
	
	\item \textit{the \textbf{r}ight-most position of the \textbf{sma}llest entry}: $\RSMA(w):=\max\{i:w_i=\sma(w)\}$. 
\end{itemize}

Now, we split $\bN$ into two disjoint types:
\begin{align*}
	\bSL&:=\{w\in\bN:\RLAR(w)\ge \RSMA(w)\},\\
	\bLS&:=\{w\in\bN:\RLAR(w)< \RSMA(w)\}.
\end{align*}
We also write $\bSL_n=\bSL\cap\bN_n$ and $\bLS_n=\bLS\cap\bN_n$.

Define trivariate generating functions
\begin{align*}
	\cN(x,u)=\cN(x,u,t)&:=\sum_{n\ge 1}x^n\sum_{w\in\bN_n(021)}u^{\lar(w)}t^{\asc(w)},\\
	\cSL(x,u)=\cSL(x,u,t)&:=\sum_{n\ge 1}x^n\sum_{w\in\bSL_n(021)}u^{\lar(w)}t^{\asc(w)},\\
	\cLS(x,u)=\cLS(x,u,t)&:=\sum_{n\ge 1}x^n\sum_{w\in\bLS_n(021)}u^{\lar(w)}t^{\asc(w)}.
\end{align*}
Note that
\begin{align}\label{eq:N-SL-LS}
	\cN(x,u)=\cSL(x,u)+\cLS(x,u).
\end{align}
Let us write the coefficients as
\begin{align*}
	\coN_{[-,\ell]}(x)=\coN_{[-,\ell]}(x,t)&:=[u^\ell]\cN(x,u),\\
	\coSL_{[-,\ell]}(x)=\coSL_{[-,\ell]}(x,t)&:=[u^\ell]\cSL(x,u),\\
	\coLS_{[-,\ell]}(x)=\coLS_{[-,\ell]}(x,t)&:=[u^\ell]\cLS(x,u).
\end{align*}

\begin{lemma}\label{le:SL(021)}
	For any $w\in\bSL_n(021)$, we have $\RLAR(w)=n$. Namely, the last entry in $w$ is the largest.
\end{lemma}

\begin{proof}
	Assume that $\RLAR(w)=j<n$. Then for any $k$ with $j<k\le n$, we have $w_k<\lar(w)$. We also claim that $w_k>\sma(w)$. Otherwise, $\RLAR(w)< \RSMA(w)$ and thus $w\not\in \bSL_n(021)$. The above also indicates that $\sma(w)<\lar(w)$. Now we assume that $w_i=\sma(w)$ for some $i$. Since $\RLAR(w)\ge \RSMA(w)$ and $\sma(w)\ne\lar(w)$, we must have $i<j$. However, the subsequence $w_iw_jw_k$ satisfies $\lar(w)=w_j>w_k>w_i=\sma(w)$ with $i<j<k$, and is therefore order isomorphic to $021$, thereby leading to a contradiction.
\end{proof}

\begin{lemma}\label{le:SL-func-eqn}
	We have
	\begin{align}\label{eq:SL-func-eqn}
		\cSL(x,u)=\frac{x}{(1-u)(1-x+xt)}+\frac{xt}{(1-u)(1-x+xt)} \cN(x,u).
	\end{align}
\end{lemma}

\begin{proof}
	Assume that $n\ge 2$. Let $w=w_1\cdots w_n\in \bSL_n(021)$. By Lemma \ref{le:SL(021)}, we have $w_n=\lar(w)$. Also, we note that the subsequence $w'=w_1\cdots w_{n-1}$ is in $\bN_{n-1}(021)$. Further, $\lar(w')\le w_n$. Hence, any $w\in \bSL_n(021)$ can be uniquely generated by a $w'\in\bN_{n-1}(021)$ appended by an entry $\ell$ no smaller than $\ell':=\lar(w')$. To keep track of the ascent statistic, we need to distinguish $w'$ depending on whether it is in $\bLS_{n-1}(021)$ or $\bSL_{n-1}(021)$.
	
	If $w'\in \bLS_{n-1}(021)$, then the last entry of $w'$ is not the largest by the definition of $\bLS$, and thus it is smaller than $\ell'$. Therefore, $\asc(w)=\asc(w')+1$. If $w'\in \bSL_{n-1}(021)$, then the last entry of $w'$ is the largest by Lemma \ref{le:SL(021)}, and thus it equals $\ell'$. Therefore, if $\ell>\ell'$, we have $\asc(w)=\asc(w')+1$; if $\ell=\ell'$, we have $\asc(w)=\asc(w')$.
	
	Noting that sequences in $\bSL_1(021)$ are generated by
	\begin{align*}
		x\sum_{\ell\ge 0}u^\ell = \frac{x}{1-u},
	\end{align*}
	we have
	\begin{align*}
		\cSL(x,u)&=\frac{x}{1-u}+\sum_{\ell'\ge 0}\coLS_{[-,\ell']}(x)\sum_{\ell\ge \ell'}xu^{\ell}t + \sum_{\ell'\ge 0}\coSL_{[-,\ell']}(x)\left(xu^{\ell'}+\sum_{\ell> \ell'}xu^{\ell}t\right)\\
		&=\frac{x}{1-u}+ \sum_{\ell'\ge 0}\coSL_{[-,\ell']}(x)u^{\ell'}(x-xt) +\sum_{\ell'\ge 0}\left(\coLS_{[-,\ell']}(x)+\coSL_{[-,\ell']}(x)\right)\sum_{\ell\ge \ell'}xu^{\ell}t\\
		&=\frac{x}{1-u}+ (x-xt)\cSL(x,u)+\frac{xt}{1-u} \cN(x,u).
	\end{align*}
	This gives the desired relation.
\end{proof}

\begin{lemma}\label{le:LS-func-eqn}
	We have
	\begin{align}\label{eq:LS-func-eqn}
		\cLS(x,u)=\cN(x,u)\left(\cSL(x,u)-\frac{x}{1-x}\right).
	\end{align}
\end{lemma}

\begin{proof}
	Let $w=w_1\cdots w_n\in \bLS_n(021)$. Here $n\ge 2$ by the definition of $\bLS$. We first write $\ell=\lar(w)$ and assume that $j=\RLAR(w)$. Since $\RLAR(w)<\RSMA(w)$, we have $j<n$. Also, for any $k$ with $j+1\le k\le n$, we have
	\begin{align}\label{eq:LS-cond-1}
		w_k<w_j=\ell.
	\end{align}
	
	Now, we split $w$ into two subsequences: $w'=w_1\cdots w_j$ and $w''=w_{j+1}\cdots w_n$. Then both $w'$ and $w''$ are non-empty and avoid the pattern $021$. We also note that the last entry in $w'$ is also the largest, and thus $w'\in\bSL(021)$.
	
	Let $s'=\sma(w')$ and assume that $w_i=s'$ for some $i$ with $1\le i\le j$. Then for any $k$ with $j+1\le k\le n$, we have
	\begin{align}\label{eq:LS-cond-2}
		w_k\le s'.
	\end{align}
	Otherwise, if $w_k>s'$, then $i<j$ and the subsequence $w_iw_jw_k$ satisfies $\ell=w_j>w_k>w_i=s'$, and is therefore order isomorphic to $021$, which is prohibited.
	
	By \eqref{eq:LS-cond-1} and \eqref{eq:LS-cond-2}, we know that $\lar(w')>\lar(w'')$ and $\sma(w')\ge \lar(w'')$. For any sequence in $\bSL(021)$ with $\lar\ge 1$ (i.e., aside from those sequences consitsted of purely $0$'s), we add $\lar(w'')$ to each entry of this sequence. Then the above $w'$ is uniquely generated. This process also preserves the ascent statistic. If we write $\lar(w'')=\ell''$, then these $w'$ are generated by
	\begin{align*}
		\sum_{\ell'\ge 1}\coSL_{[-,\ell']}(x) u^{\ell'+\ell''}&=u^{\ell''}\sum_{\ell'\ge 1}\coSL_{[-,\ell']}(x) u^{\ell'}\\
		&=u^{\ell''}\left(\cSL(x,u)-\sum_{n\ge 1}x^nu^0t^0\right)\\
		&=u^{\ell''}\left(\cSL(x,u)-\frac{x}{1-x}\right).
	\end{align*}
	
	Therefore, noting that $\asc(w)=\asc(w')+\asc(w'')$, we have
	\begin{align*}
		\cLS(x,u)&=\sum_{\ell''\ge 0}\coN_{[-,\ell'']}(x)u^{\ell''}\left(\cSL(x,u)-\frac{x}{1-x}\right)\\
		&=\cN(x,u)\left(\cSL(x,u)-\frac{x}{1-x}\right),
	\end{align*}
	which is our required result.
\end{proof}

\begin{theorem}
	We have
	\begin{align}\label{eq:N}
		\cN(x,u)=\frac{1-u-2x+ux(1-t)+x^2(1+t)}{2x(1-x)t}-\frac{\sqrt{\Delta}}{2x(1-x)t},
	\end{align}
	where
	\begin{align}
		\Delta=\big(1-u-2x+ux(1-t)+x^2(1+t)\big)^2-4(1-x)^2 x^2 t.
	\end{align}
\end{theorem}

\begin{proof}
	Joining \eqref{eq:N-SL-LS} with \eqref{eq:SL-func-eqn} gives
	\begin{equation*}
		\left\{
		\begin{aligned}
			\cSL&=\dfrac{xt}{(1-u)(1-x+xt)} \cN+\dfrac{x}{(1-u)(1-x+xt)},\\[12pt]
			\cLS&=\dfrac{1-u-x+ux-uxt}{(1-u)(1-x+xt)}\cN-\dfrac{x}{(1-u)(1-x+xt)}.
		\end{aligned}
		\right.
	\end{equation*}
	Substituting the above into \eqref{eq:LS-func-eqn} yields the following quadratic equation of $\cN$:
	\begin{align}\label{quadratic eq:N}
	tx(1-x)\cN^2-\big((t+1)x^2-(ut-u+2)x-u+1\big)\cN+x(1-x)=0.
	\end{align}
	Recalling that $\cN$ is a formal power series in $x$, $u$ and $t$, with initial condition $\cN(x,0,t)=x/(1-x)$, we only have one admissible solution of \eqref{quadratic eq:N}, as given in \eqref{eq:N}.
\end{proof}

\section{Hong and Li's conjecture with ascents}

We still start with $\bI_{n,p}$, the set of $(p,p+2,\ldots,p+n)$-inversion sequences. Let us split $\bI_{n,p}$ into two disjoint types:
\begin{align*}
	\bSL_{n,p}&:=\{e\in\bI_{n,p}:\RLAR(e)\ge \RSMA(e)\},\\
	\bLS_{n,p}&:=\{e\in\bI_{n,p}:\RLAR(e)< \RSMA(e)\}.
\end{align*}

Define quadvariate generating functions
\begin{align*}
	\cG(x,s,u)=\cG(x,s,u,t)&:=\sum_{p\ge 1}\sum_{n\ge 1}x^ns^p \sum_{e\in\bI_{n,p}(021)}u^{\lar(e)}t^{\asc(e)},\\
	\cG^*(x,s,u)=\cG^*(x,s,u,t)&:=\sum_{p\ge 1}\sum_{n\ge 1}x^ns^p \sum_{e\in\bSL_{n,p}(021)}u^{\lar(e)}t^{\asc(e)},\\
	\cG^{**}(x,s,u)=\cG^{**}(x,s,u,t)&:=\sum_{p\ge 1}\sum_{n\ge 1}x^ns^p \sum_{e\in\bLS_{n,p}(021)}u^{\lar(e)}t^{\asc(e)}.
\end{align*}
Note that
\begin{align}\label{eq:G-G*-G**}
	\cG(x,s,u)=\cG^*(x,s,u)+\cG^{**}(x,s,u).
\end{align}
We also write the coefficients as
\begin{align*}
	g^*_{[n,p,\ell]}=g^*_{[n,p,\ell]}(t)&:=[x^n s^p u^\ell]\cG^*(x,s,u),\\
	g^{**}_{[n,p,\ell]}=g^{**}_{[n,p,\ell]}(t)&:=[x^n s^p u^\ell]\cG^{**}(x,s,u),\\
	g^*_{[-,p,\ell]}(x)=g^*_{[-,p,\ell]}(x,t)&:=[s^p u^\ell]\cG^*(x,s,u),\\
	g^{**}_{[-,p,\ell]}(x)=g^{**}_{[-,p,\ell]}(x,t)&:=[s^p u^\ell]\cG^{**}(x,s,u).
\end{align*}

\begin{lemma}
	We have
	\begin{align}\label{eq:G*-func-eqn}
		\cG^*(x,s,u) &= \frac{xs}{(1-s)(1-us)(1-x+xt)}\notag\\
		&\quad+\frac{xt}{(1-u)(1-x+xt)}\big(\cG(x,s,u)-u\cG(ux,us,1)\big).
	\end{align}
\end{lemma}

\begin{proof}
	We proceed in an analogous way to the proof of Lemma \ref{le:SL-func-eqn}. For $n\ge 1$, each sequence $e$ in $\bSL_{n+1,p}(021)$ is uniquely generated by appending to a sequence $e'$ in $\bI_{n,p}(021)$ by a number $\ell$ with $\ell'=:\lar(e')\le \ell\le n+p$. If $e'\in \bLS_{n,p}(021)$, then $\asc(e)=\asc(e')+1$. If $e'\in \bSL_{n,p}(021)$, we have two subcases: if $\ell>\ell'$, then $\asc(e)=\asc(e')+1$; if $\ell=\ell'$, then $\asc(e)=\asc(e')$.
	
	Noting that sequences in $\bSL_{1,p}(021)$ are generated by
	\begin{align*}
		x\sum_{\ell=0}^{p-1}u^\ell=\frac{x(1-u^p)}{1-u},
	\end{align*}
	we have
	\begin{align*}
		\cG^*(x,s,u)&=\sum_{p\ge 1}s^p\frac{x(1-u^p)}{1-u}+\sum_{p\ge 1}s^p\sum_{n\ge 1}x^n\sum_{\ell'\ge 0}g^{**}_{[n,p,\ell']}\sum_{\ell=\ell'}^{n+p}xu^{\ell}t\\
		&\quad+\sum_{p\ge 1}s^p\sum_{n\ge 1}x^n\sum_{\ell'\ge 0}g^{*}_{[n,p,\ell']}\left(xu^{\ell'}+\sum_{\ell=\ell'+1}^{n+p}xu^{\ell}t\right)\\
		&=\sum_{p\ge 1}s^p\frac{x(1-u^p)}{1-u}+\sum_{p\ge 1}s^p\sum_{n\ge 1}x^n\sum_{\ell'\ge 0}g^{*}_{[n,p,\ell']}u^{\ell'}(x-xt)\\
		&\quad+\sum_{p\ge 1}s^p\sum_{n\ge 1}x^n\sum_{\ell'\ge 0}\left(g^{*}_{[n,p,\ell']}+g^{**}_{[n,p,\ell']}\right)\sum_{\ell=\ell'}^{n+p}xu^{\ell}t\\
		&=\frac{x}{1-u}\left(\frac{s}{1-s}-\frac{us}{1-us}\right)+(x-xt)\cG^*(x,s,u)\\
		&\quad +\frac{xt}{1-u}\big(\cG(x,s,u)-u\cG(ux,us,1)\big).
	\end{align*}
	This gives the desired relation.
\end{proof}

\begin{lemma}
	We have
	\begin{align}\label{eq:G**-func-eqn}
		\cG^{**}(x,s,u) = \cN(x,us)\left(\cG^*(x,s,u)-\frac{x}{1-x}\frac{s}{1-s}\right).
	\end{align}
\end{lemma}

\begin{proof}
	We proceed in an analogous way to the proof of Lemma \ref{le:LS-func-eqn}. For $n\ge 2$, each sequence in $\bLS_{n,p}(021)$ is uniquely generated by appending to a sequence $e'$ in $\bSL_{n',p}(021)$ (with $1\le n'\le n-1$) by a sequence $e''$ in $\bN_{n-n'}(021)$ such that $\lar(e')>\lar(e'')$ and $\sma(e')\ge \lar(e'')$. Such $e'$ can be obtained by adding $\lar(e'')$ to each entry of a sequence in $\bSL_{n',p-\lar(e'')}(021)$ with $\lar\ge 1$. If we write $\lar(e'')=\ell''$, then these $e'$ are generated by
	\begin{align*}
		\sum_{p\ge \ell''+1}s^p\sum_{\ell'\ge 1}g^*_{[-,p-\ell'',\ell']}(x)u^{\ell'+\ell''}&=u^{\ell''}s^{\ell''}\sum_{p'\ge 1}\sum_{\ell'\ge 1}g^*_{[-,p',\ell']}(x)u^{\ell'}s^{p'}\\
		%&=u^{\ell''}t^{\ell''}\left(\sum_{p'\ge 1}\sum_{\ell'\ge 0}g^*_{[-,\ell',p']}(x)u^{\ell'}t^{p'}-\frac{x}{1-x}\frac{t}{1-t}\right)\\
		&=u^{\ell''}s^{\ell''}\left(\cG^*(x,s,u)-\frac{x}{1-x}\frac{s}{1-s}\right).
	\end{align*}
	
	Therefore, noting that $\asc(e)=\asc(e')+\asc(e'')$, we have
	\begin{align*}
		\cG^{**}(x,s,u)&=\sum_{\ell''\ge 0}\coN_{[-,\ell'']}(x)u^{\ell''}s^{\ell''}\left(\cG^*(x,s,u)-\frac{x}{1-x}\frac{s}{1-s}\right)\\
		&=\cN(x,us)\left(\cG^*(x,s,u)-\frac{x}{1-x}\frac{s}{1-s}\right),
	\end{align*}
	which is our required result.
\end{proof}

\begin{theorem}
	We have
	\begin{align}\label{eq:F}
		\cG(x,xt,1)=\frac{t\big((1+x-xt) r-x\big)}{(1-r)(1-xt)},
	\end{align}
	where
	\begin{align}\label{eq:x-expansion}
		r&=x + t x^2 + (t ^2+2 t) x^3 + (t ^3+8 t ^2+3 t) x^4\notag\\
		&\quad + (t ^4+22 t ^3+27 t ^2+4 t) x^5+\cdots
	\end{align}
	is the unique power series solution (with $r\to 0$ as $x\to 0$) to
	\begin{align}\label{eq:kernel-poly}
		r^3-(2+x+t-xt^2)r^2+(1+2x)r-x=0.
	\end{align}
\end{theorem}

\begin{proof}
	For convenience, we define
	\begin{align*}
		\cF(x,u)&:=\cG(x,xt,u),\\
		\cF^*(x,u)&:=\cG^*(x,xt,u),\\
		\cF^{**}(x,u)&:=\cG^{**}(x,xt,u).
	\end{align*}
	Combining \eqref{eq:G-G*-G**} and \eqref{eq:G*-func-eqn} with $s\mapsto xt$ gives
	\begin{align*}
		\left\{\begin{aligned}
			\cF^*(x,u)&=\dfrac{x^2t}{(1-xt)(1-uxt)(1-x+xt)}+\dfrac{xt}{(1-u)(1-x+xt)}\cF(x,u)\\
			&\quad-\dfrac{uxt}{(1-u)(1-x+xt)}\cF(ux,1),\\[12pt]
			\cF^{**}(x,u)&=-\dfrac{x^2t}{(1-xt)(1-uxt)(1-x+xt)}+\dfrac{1-u-x+ux-uxt}{(1-u)(1-x+xt)}\cF(x,u)\\
			&\quad+\dfrac{uxt}{(1-u)(1-x+xt)}\cF(ux,1).\\
		\end{aligned}\right.
	\end{align*}
	Now, it follows by substituting the above into \eqref{eq:G**-func-eqn} with $s\mapsto xt$ that
%	\begin{equation*}
%		\scalebox{.875}{%
%			$\begin{aligned}
%				&\left(\dfrac{1-u-x+ux-uxt}{(1-u)(1-x+xt)}-\dfrac{xt}{(1-u)(1-x+xt)}\cN(x,uxt)\right)\cF(x,u)\\
%				&\qquad=\dfrac{x^2t}{(1-xt)(1-uxt)(1-x+xt)}\left(1-\frac{xt(1-u+ux-uxt)}{1-x}\cN(x,uxt)\right)\\
%				&\qquad\quad-\left(\dfrac{uxt}{(1-u)(1-x+xt)}+\dfrac{uxt}{(1-u)(1-x+xt)}\cN(x,uxt)\right)\cF(ux,1).
%			\end{aligned}$}
%	\end{equation*}
	\begin{align*}
		&\left(\dfrac{1-u-x+ux-uxt}{(1-u)(1-x+xt)}-\dfrac{xt}{(1-u)(1-x+xt)}\cN(x,uxt)\right)\cF(x,u)\\
		&\qquad=\dfrac{x^2t}{(1-xt)(1-uxt)(1-x+xt)}\left(1-\frac{xt(1-u+ux-uxt)}{1-x}\cN(x,uxt)\right)\\
		&\qquad\quad-\left(\dfrac{uxt}{(1-u)(1-x+xt)}+\dfrac{uxt}{(1-u)(1-x+xt)}\cN(x,uxt)\right)\cF(ux,1).
	\end{align*}
	We then make the change of variables $u\mapsto w/x$. Therefore,
	\begin{align}\label{eq:F-kernel}
		K\cdot \cF(x,x^{-1}w)=C_0-C_1\cdot \cF(w,1),
	\end{align}
	where
%	\begin{equation*}
%		\scalebox{.875}{%
%			$\begin{aligned}
%				K&=\dfrac{1-x^{-1}w-x+w-wt}{(1-x^{-1}w)(1-x+xt)}-\dfrac{xt}{(1-x^{-1}w)(1-x+xt)}\cN(x,wt),\\
%				C_0&=\dfrac{x^2t}{(1-xt)(1-wt)(1-x+xt)}\left(1-\frac{xt(1-x^{-1}w+w-wt)}{1-x}\cN(x,wt)\right),\\
%				C_1&=\dfrac{wt}{(1-x^{-1}w)(1-x+xt)}+\dfrac{wt}{(1-x^{-1}w)(1-x+xt)}\cN(x,wt).
%			\end{aligned}$}
%	\end{equation*}
	\begin{align*}
		K&=\dfrac{1-x^{-1}w-x+w-wt}{(1-x^{-1}w)(1-x+xt)}-\dfrac{xt}{(1-x^{-1}w)(1-x+xt)}\cN(x,wt),\\
		C_0&=\dfrac{x^2t}{(1-xt)(1-wt)(1-x+xt)}\left(1-\frac{xt(1-x^{-1}w+w-wt)}{1-x}\cN(x,wt)\right),\\
		C_1&=\dfrac{wt}{(1-x^{-1}w)(1-x+xt)}+\dfrac{wt}{(1-x^{-1}w)(1-x+xt)}\cN(x,wt).
	\end{align*}
	Also, we recall from \eqref{eq:N},
	\begin{align}\label{eq:N-w}
		\cN(x,wt)=\frac{1-wt-2x+wxt(1-t)+x^2(1+t)}{2x(1-x)t}-\frac{\sqrt{\Delta^*}}{2x(1-x)t},
	\end{align}
	where
	\begin{align*}
		\Delta^*=\big(1-wt-2x+wxt(1-t)+x^2(1+t)\big)^2-4(1-x)^2 x^2 t.
	\end{align*}
	Applying the kernel method to \eqref{eq:F-kernel} by setting the kernel polynomial $K$ to be zero, we have
	\begin{align*}
		\dfrac{1-x^{-1}w-x+w-wt}{(1-x^{-1}w)(1-x+xt)}-\dfrac{xt}{(1-x^{-1}w)(1-x+xt)}\cN(x,wt)=0,
	\end{align*}
	or,
	\begin{align}\label{eq:ker-poly-1}
		\cN(x,wt) = \frac{1-x^{-1}w-x+w-wt}{xt},
	\end{align}
	or by recalling \eqref{eq:N-w},
	\begin{align}\label{eq:ker-poly-2}
		x^3-(2+w+t-wt^2)x^2+(1+2w)x-w=0.
	\end{align}
	Our admissible solution $x=x(w)$ satisfies $x\to 0$ as $w\to 0$, and has the power series expansion
	\begin{align*}
		x&=w + t w^2 + (t ^2+2 t) w^3 + (t ^3+8 t ^2+3 t) w^4\\
		&\quad + (t ^4+22 t ^3+27 t ^2+4 t) w^5+\cdots.
	\end{align*}
	On the other hand,
	\begin{align*}
		\cF(w,1)&=\frac{C_0}{C_1}\\
		\text{\tiny (by \eqref{eq:ker-poly-1})}&=-\frac{t \big(x^3+(wt-w-2)x^2+w x\big)}{(1-x)(1-xt)(1-wt)}\\
		\text{\tiny (by \eqref{eq:ker-poly-2})} &= \frac{t\big((1+w-wt) x-w\big)}{(1-x)(1-wt)}.
	\end{align*}
	Finally, the desired result follows by renaming the variables $(w,x)\mapsto (x,r)$.
\end{proof}

Let
\begin{align}
	\cI(x,t):=\sum_{n\ge 1}x^n\sum_{e\in\bI_n(0021)}t^{\asc(e)}.
\end{align}
We are in a position to establish the following ``explicit'' expression for $\cI(x,t)$.

\begin{theorem}
	We have
	\begin{align}\label{eq:A-r}
		\cI(x,t) = \frac{r}{1-r},
	\end{align}
	where $r$ is as in \eqref{eq:x-expansion}.
\end{theorem}

\begin{proof}
	Recall that any inversion sequence $e$ in $\bI_n(0021)$ is either in correspondence with a sequence $\hat{e}$ in $\bI_{n-p,p}(021)$ for some $p$ by Lemma \ref{lem:0021}, in which case $\asc(e)=(p-1)+\asc(\hat{e})$, or a sequence with only tight entries, which is generated by
	$$x+t x^2+t^2 x^3+\cdots = \frac{x}{1-xt}.$$
	Hence,
	\begin{align*}
		\cI(x,t)&=\frac{x}{1-xt}+\sum_{p\ge 1}\sum_{n\ge 1}x^{n+p}t^{p-1} \sum_{e\in\bI_{n,p}(021)}t^{\asc(e)}\\
		&=\frac{x}{1-xt}+t^{-1}\cG(x,xt,1).
	\end{align*}
	The desired result follows by recalling \eqref{eq:F}.
\end{proof}

Now, we conclude our proof of \eqref{eq:I-gf}.

\begin{proof}[Proof of \eqref{eq:I-gf}]
	By \eqref{eq:A-r}, we have
	\begin{align*}
		r=\frac{\cI(x,t)}{\cI(x,t)+1}.
	\end{align*}
	Substituting the above into \eqref{eq:kernel-poly} gives
	\begin{align*}
		0&=\left(\frac{\cI(x,t)}{\cI(x,t)+1}\right)^3-(2+x+t-xt^2)\left(\frac{\cI(x,t)}{\cI(x,t)+1}\right)^2\\
		&\quad+(1+2x)\left(\frac{\cI(x,t)}{\cI(x,t)+1}\right)-x,
	\end{align*}
	which finally results in
	\begin{align*}
		\cI(x,t)=\big(1+\cI(x,t)\big)\big(x+t \cI(x,t)^2 - xt^2 \cI(x,t)^2\big).
	\end{align*}
	This functional equation is identical to \eqref{eq:A(x,t)-eqn}.
\end{proof}

\section{Conclusion}

Burstein's second and third conjectures remain open. It seems that the $\alt$ statistic may still be a key as one can easily modify the proof of Lemma \ref{lem:alt-bound} to obtain the following analogous result.

\begin{lemma}
	For each permutation $\pi$ in
	$$\fS_n(2134, 42153, 24153)\quad\text{or}\quad\fS_n(2143, 42135, 24135),$$
	we have $\alt(\pi)\le 4$.
\end{lemma}

However, in these two cases, there is an obstacle to deducing parallel relations to those in \eqref{id:LandR}, especially for $k=2$ and $3$.

Finally, for $\pi$ a permutation of length $n$, we define
\begin{itemize}[leftmargin=*,align=left]
	\renewcommand{\labelitemi}{$\triangleright$}
	\item \textit{the number of \textbf{i}nverse \textbf{asc}ents}: $\iasc(\pi):=\asc(\pi^{-1})$ where $\pi^{-1}$ is the inverse permutation of $\pi$;
	
	\item \textit{the number of \textbf{r}ight-to-left \textbf{ma}xima}: $\rma(\pi):=|\{i\in[n]: \pi_j<\pi_i,\, \forall j>i\}|$; 
	
	\item \textit{the set of positions of \textbf{l}eft-to-right \textbf{mi}nima}: $\LMI(\pi):=\{i\in[n]: \pi_j>\pi_i,\, \forall j<i\}$.
\end{itemize}

Our numerical evidence supports the following equidistribution conjectures.

\begin{conjecture}
	The triple of statistics $(\LMI,\rma,\ides)$ has the same distribution over $\fS_n(2134,42153,24153)$ and $\fS_n(3124,42153,24153)$. 
\end{conjecture}

\begin{conjecture}
	The pair of Eulerian statistics $(\des,\ides)$ over  $\fS_n(2134,42153,\allowbreak 24153)$ has the same distribution as $(\asc,\iasc)$ over $\fS_n(2143,42135,24135)$. 
\end{conjecture}

\subsection*{Acknowledgements}

S.~Chern was supported by a Killam Postdoctoral Fellowship from the Killam Trusts.
S.~Fu was supported by the National Natural Science Foundation of China grant 12171059 and the Natural Science Foundation Project of CQ CSTC (No.~cstc2021jcyj-msxmX0693).
 Z.~Lin was supported by the National Natural Science Foundation of China grant 12271301 and  the project of Qilu Young Scholars of Shandong University.

\bibliographystyle{amsplain}

\end{document}